\documentclass[11pt]{amsart}
\usepackage{amscd,amsfonts,amssymb,amsmath}
\usepackage{hyperref}
\usepackage{tikz}
\usepackage{latexsym}
\usepackage{amssymb}
\usepackage{amsthm}
\usepackage{tikz}
\usepackage{amscd}
\usepackage[all,cmtip]{xy}
\newtheorem{theorem}{Theorem}[section]

\newtheorem{proposition}[theorem]{Proposition}
\newcommand{\End}{\operatorname{End} }
\newcommand{\Ext}{\operatorname{Ext} }
\newcommand{\Aut}{\operatorname{Aut} }
\newcommand{\C}{\operatorname{C} }
\newcommand{\Hom}{\operatorname{Hom} }
\newcommand{\Ha}{\operatorname{H} }
\newcommand{\B}{\operatorname{B} }
\newcommand{\Z}{\operatorname{Z} }
\newcommand{\im}{\operatorname{image} }

\theoremstyle{definition}

\newtheorem{example}[theorem]{Example}
\newtheorem{remark}[theorem]{Remark}
\numberwithin{equation}{subsection}
\theoremstyle{plain}

\newtheorem{problem}{Problem}
\newtheorem{corollary}[theorem]{Corollary}

\numberwithin{equation}{section}

\begin{document}
\title{Extensions and automorphisms of Lie algebras}
\author{Valeriy G. Bardakov}
\author{Mahender Singh}
\date{\today}
\address{Sobolev Institute of Mathematics,  Novosibirsk 630090, Russia}
\address{Novosibirsk State University, Novosibirsk 630090, Russia}
\address{Laboratory of Quantum Topology, Chelyabinsk State University, Brat'ev Kashirinykh street 129, Chelyabinsk 454001, Russia.}
\address{Novosibirsk State Agrarian University, Dobrolyubova street, 160,  Novosibirsk, 630039, Russia.}

\email{bardakov@math.nsc.ru}
\address{Indian Institute of Science Education and Research (IISER) Mohali, Sector 81,  S. A. S. Nagar, P. O. Manauli, Punjab 140306, India.}
\email{mahender@iisermohali.ac.in}


\begin{abstract}
Let $0 \to A \to L \to B \to 0$ be a short exact sequence of Lie algebras over a field $F$, where $A$ is abelian. We show that the obstruction for a pair of automorphisms in $\Aut(A) \times \Aut(B)$ to be induced by an automorphism in $\Aut(L)$ lies in the Lie algebra cohomology $\Ha^2(B;A)$. As a consequence, we obtain a four term exact sequence relating automorphisms, derivations and cohomology of Lie algebras. We also obtain a more explicit necessary and sufficient condition for a pair of automorphisms in $\Aut\big(L_{n,2}^{(1)}\big) \times \Aut\big(L_{n,2}^{ab}\big)$ to be induced by an automorphism in $\Aut\big(L_{n,2}\big)$, where $L_{n,2}$ is a free nilpotent Lie algebra of rank $n$ and step $2$.
\end{abstract}
\maketitle

\section{Introduction} \label{sec1}
Let $A$ and $B$ be Lie algebras over a field $F$ where $A$ is abelian. We say that $A$ is a left $B$-module if there is a $F$-homomorphism $B \otimes A \to A$ written as $b \otimes a \mapsto ba$ such that $$[b_1,b_2]a=b_1(b_2a)-b_2(b_1a)~\textrm{for all}~b_1,b_2 \in B~\textrm{and}~a \in A.$$ Let $\End(A)$ be the Lie algebra of all $F$-endomorphisms of $A$ equipped with the Lie bracket $$[f,g]=f g- g f~\textrm{for all}~f,g \in \End(A).$$ Here $f g (a) = f\big(g(a)\big)$ for $a \in A$. Then a left $B$-module structure on $A$ is equivalent to existence of a Lie algebra homomorphism $$B \to  \End(A).$$

Let $A$ and $B$ be Lie algebras. Then an extension of $B$ by $A$ is a short exact sequence of Lie algebras $$0 \to A \stackrel{i}{\to} L \stackrel{p}{\to} B \to 0,$$ where $L$ is a Lie algebra. Without loss of generality, we may assume that $i$ is the inclusion map and we omit it from the notation. It follows from the exactness that $A$ is an ideal of $L$. This together with the Jacobi identity gives a left $L$-module structure on $A$ given by $$xa= [x,a]~\textrm{for}~ x \in L~\textrm{and}~ a \in A.$$  Let $s:B \to L$ be a section of $p$, that is, $s$ is a $F$-linear map such that $p s=1$. If $A$ is abelian, then this induces a left $B$-module structure on $A$ given by $$ba=\big[s(b),a\big]~\textrm{for}~b \in B~\textrm{and}~a \in A.$$ In fact, the converse is also true. Note that the above $B$-module structure on $A$ does not depend on the choice of the section. We denote the above $B$-module structure on $A$ by $$\alpha: B \to  \End(A).$$

Let $\Aut(A)$, $\Aut(L)$ and $\Aut(B)$ denote the groups of all Lie algebra automorphisms of $A$, $L$ and $B$, respectively.  Let $\Aut_A(L)$ denote the group of all Lie algebra automorphisms of $L$ which keep $A$ invariant as a set. Note that an automorphism $\gamma \in \Aut_A(L)$ induces automorphisms $\gamma|_A \in \Aut(A)$ and $\overline{\gamma} \in \Aut(B)$ given by $\gamma|_A(a)= \gamma(a)$ for all $a \in A$ and $\overline{\gamma}(b)= p \big(\gamma(s(b))\big)$ for all $b \in B$. This gives a group homomorphism $$\tau: \Aut_A(L) \to \Aut(A) \times \Aut(B)$$ given by $$\tau(\gamma)= (\gamma|_A, \overline{\gamma}).$$ 

A pair of automorphisms $(\theta, \phi) \in \Aut(A)\times \Aut(B)$ is called inducible if there exists a $\gamma \in \Aut_A(L)$ such that $\tau(\gamma)= (\theta, \phi)$. Our main aim in this paper is to investigate the following natural problem.

\begin{problem}\label{prob1}
Let $0 \to A \to L \to B \to 0$ be an extension of Lie algebras over a field $F$. Under what conditions is a pair of automorphisms $(\theta, \phi) \in \Aut(A) \times \Aut(B)$ inducible?
\end{problem}

Various aspects of automorphisms of Lie algebras have been investigated extensively in the literature, see for example \cite{Borel, Jacobson, Steinberg}. Automorphisms of real Lie algebras of dimension five or less have been classified in \cite{Fisher}. Further, this work has been extended to automorphisms of six dimensional real Lie algebras in the recent thesis by Gray \cite{Gray}. Recall that, the derived series of a Lie algebra $L$ is given by $L^{(0)} = L$ and $L^{(k)} = [L^{(k-1)}, L^{(k-1)}]$ for $k \geq 1$. In \cite{Bahturin}, Bahturin and Nabiyev investigated the structure of automorphism groups of Lie algebras of the form $L/[R,R]$, where  $L$ a free Lie algebra over a commutative and associative ring $k$ and $R$ is an ideal of $L$ such that $L/R$ is a free $k$-module. As an  application, they obtained the structure of the automorphism group of free metabelian Lie algebra $L/L^{(2)}$ of finite rank, and showed that it has automorphisms which cannot be lifted to automorphisms of $L$. However, to our knowledge, almost nothing seems to be known about Problem \ref{prob1}. For extensions of groups, the analogous problem has been investigated recently in \cite{Jin, Jin1, Passi, Robinson}, wherein using cohomological methods the problem for finite groups has been reduced to $p$-groups.

We would like to remark that our approach in this paper is purely algebraic and we make no reference to Lie groups.  However, it would be interesting to explore connections of our results to Lie groups, and we leave it to the curious reader. 

The paper is organized as follows. In Section \ref{sec2}, we obtain a necessary and sufficient condition for a pair of automorphisms to be inducible. Our main theorem is the following.

\begin{theorem}\label{main-thm}
Let $0 \to A \to L \to B \to 0$ be an abelian extension of Lie algebras over a field $F$ and $(\theta, \phi) \in \Aut(A)\times \Aut(B)$. Then the pair $(\theta, \phi)$ is inducible if and only if the following two conditions hold:
\begin{enumerate}
\item There exists an element $\lambda \in \Hom(B,A)$ such that 
$$\theta \big( \mu(b_1 , b_2)\big)-\mu \big(\phi(b_1),\phi(b_2) \big)=\phi(b_1)\lambda(b_2)-\phi(b_2)\lambda(b_1)-\lambda \big([b_1,b_2] \big)~\textrm{for all}~b_1,b_2 \in B.$$
\item $\alpha \phi(b) =\theta \alpha(b) \theta^{-1}$ for all $b \in B$.
\end{enumerate}
\end{theorem}

Here, $\Hom(B,A)$ is the group of all $F$-linear maps from $B$ to $A$ and $\mu$ is the 2-cocycle corresponding to the extension $0 \to A \to L \to B \to 0$ (see Section \ref{sec2} for details). In Section \ref{sec3}, we use this condition to obtain an obstruction to inducibility of automorphisms and derive an exact sequence relating automorphisms, derivations and cohomology of Lie algebras. More precisely, we prove the following theorem.

\begin{theorem}\label{wells-sequence}
Let $\mathcal{E}:0 \to A \to L \to B \to 0$ be an abelian extension of Lie algebras over a field $F$. Let $\alpha: B \to  \End(A)$ denote the induced $B$-module structure on $A$. Then there is an exact sequence
$$0 \to \Z^1(B;A) \stackrel{i}{\longrightarrow} \Aut_A(L) \stackrel{\tau}{\longrightarrow} C_ \alpha \stackrel{\omega_{\mathcal{E}}}{\longrightarrow} \Ha^2(B;A).$$
\end{theorem}

See Section \ref{sec3} for unexplained notation. Finally, in Section \ref{sec4}, we discuss Problem \ref{prob1} for free nilpotent Lie algebras of rank $n$ and step 2.
\bigskip

\section{Condition for extension and lifting of automorphisms} \label{sec2}
Let $0 \to A \to L \stackrel{p}{\to} B \to 0$ be an abelian extension of Lie algebras and $s:B \to L$ be a section. For any two elements $b_1,b_2 \in B$, we have $$p\big(s([b_1,b_2])\big)= [b_1,b_2] = \big[p(s(b_1)), p(s(b_2))\big]= p\big( [s(b_1),s(b_2)]\big).$$ Thus there exists a unique element, say $\mu(b_1, b_2) \in A$, such that $$\mu(b_1, b_2)= \big[s(b_1),s(b_2)\big]-s\big([b_1,b_2]\big).$$ Note that  $\mu$ gives a $F$-bilinear map from
$B \times B$ to $A$ such that $\mu(b, b) = 0$ for all $b \in B$. Hence $\mu$ determines an element in $\Hom(\Lambda^2 B, A)$, which we again denote by $\mu$. 

Next, we give a necessary and sufficient condition for a pair of automorphisms to be inducible. This provides a solution of Problem \ref{prob1} when $A$ is an abelian Lie algebra.
\bigskip

\noindent \textit{Proof of Theorem \ref{main-thm}.}
Suppose that there exists a $\gamma \in \Aut_A(L)$ such that $\tau(\gamma)=(\theta, \phi)$. Let $s: B \to L$ be a section. Then any element of $L$ can be written uniquely as $a+s(b)$ for some $a\in A$ and $b \in B$. Now, $\phi(b)=p \gamma s(b)$. This implies $ps \phi(b)=p \gamma s(b)$, and hence $\gamma s(b)=s \phi(b)+ \lambda(b)$ for some $\lambda(b) \in A$. Since all the maps involved are $F$-linear, it follows that $\lambda \in \Hom(B,A)$.

To derive condition (1), we use the fact that $\gamma$ is a Lie algebra homomorphism. Let $l_1=a_1+s(b_1)$ and $l_2=a_2+s(b_2)$ be two elements of $L$. Note that $\gamma\big([l_1,l_2]\big)=\big[\gamma(l_1), \gamma(l_2)\big]$. First, consider
\begin{eqnarray*}
\gamma\big([l_1,l_2]\big) & = & \gamma\big([a_1+s(b_1), a_2+s(b_2)]\big)\\
& = & \gamma\big([s(b_1), a_2]-[s(b_2), a_1]+[s(b_1), s(b_2)]\big)\\
& = & \gamma\big([s(b_1), a_2]-[s(b_2), a_1]+\mu(b_1 , b_2) +s([b_1,b_2])\big)\\
& = & \gamma\big([s(b_1), a_2]\big)-\gamma\big([s(b_2), a_1]\big)+\gamma\big(\mu(b_1 , b_2)\big) +\gamma\big(s([b_1,b_2])\big)\\
& = & \big[\gamma(s(b_1)), \gamma(a_2)\big]-\big[\gamma(s(b_2)), \gamma(a_1)\big]+\theta\big(\mu(b_1 , b_2)\big) +s\big(\phi([b_1,b_2])\big)+\lambda([b_1,b_2])\\
& = & \big[s(\phi(b_1))+\lambda(b_1), \theta(a_2)\big]-\big[s(\phi(b_2))+\lambda(b_2), \theta(a_1)\big]\\
& & +\theta\big(\mu(b_1 , b_2)\big) +s\big([\phi(b_1),\phi(b_2)]\big)+\lambda([b_1,b_2])\\
& = & \big[s(\phi(b_1)), \theta(a_2)\big]-\big[s(\phi(b_2)), \theta(a_1)\big]+\theta\big(\mu(b_1 , b_2)\big) +s\big([\phi(b_1),\phi(b_2)]\big)+\lambda([b_1,b_2]).
\end{eqnarray*}

Next, consider
\begin{eqnarray*}
\big[\gamma(l_1),\gamma(l_2)\big] & = & \big[\gamma(a_1+s(b_1)),\gamma(a_2+s(b_2))\big]\\
& = & \big[\theta(a_1)+s(\phi(b_1))+ \lambda(b_1),\theta(a_2)+s(\phi(b_2))+ \lambda(b_2)\big]\\
& = & \big[\theta(a_1), s(\phi(b_2)) \big] + \big[s(\phi(b_1)),\theta(a_2)\big] +\big[s(\phi(b_1)),s(\phi(b_2))\big]\\
& & + \big[s(\phi(b_1)),\lambda(b_2)\big] + \big[\lambda(b_1), s(\phi(b_2))\big]\\
& = & \big[s(\phi(b_1)),\theta(a_2)\big]- \big[s(\phi(b_2)), \theta(a_1)\big] +\big[s(\phi(b_1)),s(\phi(b_2))\big]\\
&   & + \big[s(\phi(b_1)),\lambda(b_2)\big] - \big[s(\phi(b_2)), \lambda(b_1)\big]\\
& = & \big[s(\phi(b_1)),\theta(a_2)\big]- \big[s(\phi(b_2)), \theta(a_1)\big] +\mu \big(\phi(b_1), \phi(b_2)\big)\\
&   & + s\big([\phi(b_1), \phi(b_2)]\big) +\big[s(\phi(b_1)),\lambda(b_2)\big] - \big[s(\phi(b_2)), \lambda(b_1)\big]\\
& = & \big[s(\phi(b_1)),\theta(a_2)\big]- \big[s(\phi(b_2)), \theta(a_1)\big] +\mu \big(\phi(b_1), \phi(b_2)\big)\\
&   & + s\big([\phi(b_1), \phi(b_2)]\big) +\phi(b_1)\lambda(b_2) - \phi(b_2) \lambda(b_1).
\end{eqnarray*}
By comparing the LHS and the RHS we obtain condition (1).

To derive condition (2), we use the fact that $\gamma$ is an isomorphism. Let $b \in B$ and $a \in A$. Then
\begin{eqnarray*}
\alpha \big(\phi(b)\big)(a) & = & \big[s(\phi(b)),a\big]\\
& = & \big[\gamma(s(b))-\lambda(b),a\big]\\
& = & \big[\gamma(s(b)),a\big]\\
& = & \big[\gamma(s(b)),\gamma(\gamma^{-1}(a))\big]\\
& = & \gamma\big([s(b),\gamma^{-1}(a)]\big)\\
& = & \theta\big([s(b),\theta^{-1}(a)]\big)\\
& = & \theta\big(\alpha(b)(\theta^{-1}(a))\big)\\
& = & \big(\theta\alpha(b)\theta^{-1}\big)(a).
\end{eqnarray*}

Conversely, suppose that conditions (1) and (2) are given. Let $l=a+s(b) \in L$. Then define $\gamma:L \to L$ by $$\gamma(l)=\theta(a)+\lambda(b)+s(\phi(b)).$$ Since all the maps are $F$-linear, it follows that $\gamma$ is $F$-linear. Clearly, $\gamma(a)=\theta(a)$ for all $a\in A$. Further, $\overline{\gamma}(b)=p \big(\gamma(s(b))\big)=p\big( \lambda(b)+s(\phi(b))\big)=p\big(s(\phi(b))\big)= \phi(b)$ for all $b \in B$.

Suppose that $\gamma(a+s(b))=0$. This implies that $s(\phi(b))=0$. Since $s$ and $\phi$ are both injective, it follows that $b=0$. This further implies that $a=0$, and hence $\gamma$ is injective. Let $a+s(b) \in L$. Taking $a'=\theta^{-1}\big(a-\lambda(\phi^{-1}(b))\big)$ and $b'=\phi^{-1}(b)$ yields
\begin{eqnarray*}
\gamma\big(a'+s(b')\big) &=&\theta(a')+\lambda(b')+s(\phi(b'))\\
&=& \theta\big(\theta^{-1}\big(a-\lambda(\phi^{-1}(b))\big)\big)+\lambda\big(\phi^{-1}(b)\big)+s\big(\phi(\phi^{-1}(b))\big)\\
&=& a-\lambda(\phi^{-1}(b))+\lambda\big(\phi^{-1}(b)\big)+s(b)\\
&=& a+s(b).
\end{eqnarray*}
Hence $\gamma$ is surjective.

Let $l_1=a_1+s(b_1)$ and $l_2=a_2+s(b_2)$ be two elements of $L$. It remains to prove that $\gamma([l_1,l_2])=[\gamma(l_1), \gamma(l_2)]$.
\begin{eqnarray*}
\big[\gamma(l_1),\gamma(l_2)\big] & = & \big[\gamma(a_1+s(b_1)),\gamma(a_2+s(b_2))\big]\\
& = & \big[\theta(a_1)+ \lambda(b_1)+ s(\phi(b_1)),\theta(a_2)+ \lambda(b_2)+s(\phi(b_2))\big]\\
& = & \big[\theta(a_1), s(\phi(b_2)) \big] + \big[s(\phi(b_1)),\theta(a_2)\big] +\big[s(\phi(b_1)),s(\phi(b_2))\big]\\
& & + \big[s(\phi(b_1)),\lambda(b_2)\big] + \big[\lambda(b_1), s(\phi(b_2))\big]\\
& = & \theta\big( [s(b_1),a_2] \big)- \theta\big([s(b_2),a_1]\big) +\mu \big(\phi(b_1), \phi(b_2) \big)+s\big([\phi(b_1),\phi(b_2)]\big)\\
&   & + \phi(b_1)\lambda(b_2) - \phi(b_2)\lambda(b_1),~\textrm{using (2)}\\
& = & \theta\big( [s(b_1),a_2] \big)- \theta\big([s(b_2),a_1]\big) +s\big([\phi(b_1),\phi(b_2)]\big)\\
&   & +\theta \big( \mu(b_1 , b_2)\big) +\lambda([b_1,b_2]),~\textrm{using (1)}\\
&=& \theta\big([s(b_1), a_2] - [s(b_2), a_1] + \mu(b_1 , b_2)\big) +\lambda([b_1,b_2])+ s\big(\phi([b_1, b_2])\big)\\
&=& \gamma\big([s(b_1), a_2] - [s(b_2), a_1] + \mu(b_1 , b_2) + s([b_1, b_2])\big)\\
&=& \gamma\big([a_1, s(b_2)]+ [s(b_1), a_2]+ [s(b_1), s(b_2)]\big)\\
&=& \gamma\big([a_1+s(b_1),a_2+s(b_2)]\big)\\
&=& \gamma\big([l_1,l_2]\big).
\end{eqnarray*}
This completes the proof of the theorem. $\Box$

\begin{remark}\label{rem-cocycle}
To set notation, we briefly recall the definition of cohomology of Lie algebras. Let $B$ be a Lie algebra and  $A$ be a left $B$-module. For each $0 \leq k \leq \dim (B)$, define $\C^k(B;A)= \Hom(\Lambda^k B,A)$ and $\partial^k: \C^k(B;A) \to \C^{k+1}(B;A)$ by
\begin{eqnarray*}
\partial^k(\nu)(b_0,\dots,b_k) & = & \sum_{i=0}^k (-1)^i ~b_i ~\nu(\dots,\hat{b_i},\dots)\\
& & + \sum_{0 \leq i <j \leq k} (-1)^{i+j}~ \nu\big([b_i,b_j], \dots, \hat{b_i}, \dots,\hat{b_j}, \dots\big)
\end{eqnarray*}
for all $\nu \in \C^k(B; A)$. It is straightforward to verify, using the Jacobi Identity and $B$-action on $A$, that $\partial^{k+1}\partial^k=0$. Let $\Z^k(B; A)=\ker(\partial^k)$ be the group of $k$-cocycles and $\B^k(B; A)= \im(\partial^{k-1})$ be the group of $k$-coboundaries. Then $\Ha^k(B;A)= \Z^k(B; A)/\B^k(B; A)$ is the $k$-dimensional Lie algebra cohomology of $B$ with values in $A$.

Let $\alpha: B \to  \End(A)$ be the $B$-module structure on $A$ and $\Ext_\alpha(B, A)$ denote the set of equivalence classes of extensions of $B$ by $A$ inducing $\alpha$. Let $\mathcal{E}: 0 \to A \to L \to B \to 0$ be an extension inducing $\alpha$ and $\mu \in \Hom(\Lambda^2 B, A)$ be the $F$-bilinear map associated to section $s:B \to L$. Then it is easy to see that $\mu$ is a 2-cocycle and 2-cocyles corresponding to different sections differ by a 2-coboundary. Thus the map $[\mathcal{E}] \longmapsto [\mu]$ gives a bijection $$\Ext_\alpha(B, A) \longleftrightarrow \Ha^2(B; A).$$ See \cite[p.238]{Hilton} for a proof and more details.

Now, let $(\theta, \phi) \in \Aut(A) \times \Aut(B)$. Since $\phi \in \Aut(B)$, we replace $\phi(b_i)$ by $b_i$ in condition (1) and obtain the following condition:

\begin{eqnarray*}
\theta \big( \mu(\phi^{-1}(b_1), \phi^{-1}(b_2))\big)-\mu \big(b_1,b_2 \big) & = & b_1\lambda\big(\phi^{-1}(b_2)\big)-b_2\lambda\big(\phi^{-1}(b_1)\big)-\lambda\big([\phi^{-1}(b_1),\phi^{-1}(b_2)]\big)\\
& = & b_1\lambda'(b_2)-b_2\lambda'(b_1)-\lambda'\big([b_1,b_2]\big)\\
&= & \partial^1(\lambda')(b_1,b_2),
\end{eqnarray*}
where $\lambda'=\lambda \phi^{-1} \in \C^1(B;A)$ and $b_1,b_2 \in B$. Thus condition (1) is equivalent to saying that the LHS of the above equation is a 2-coboundary.
\end{remark}

\begin{remark}\label{rem-compatible}
A pair $(\theta,\, \phi) \in \Aut(A) \times \Aut(B)$ is called compatible if $$\alpha \phi(b)=\theta\alpha(b)\theta^{-1}$$ for all $b\in B$. Equivalently, the following diagram commutes.
$$
\xymatrix{
B \ar[d]_{\alpha} \ar[r]^{\phi} & B \ar[d]^{\alpha}\\
\End(A)\hspace*{3mm} \ar[r]_{f \mapsto \theta f \theta^{-1}} & \hspace*{3mm}\End(A)\\}
$$
It is easy to see that the set $C_\alpha$ of all compatible pairs is a subgroup of $\Aut(A) \times \Aut(B)$. Condition (2) of the above theorem shows that every inducible pair is compatible.

Note that $\alpha \phi$ also gives a $B$-module structure on $A$. Then the compatibility condition is equivalent to saying that $\theta: A \to A$ is a $B$-module homomorphism from the $B$-module structure $\alpha$ to the $B$-module structure $\alpha \phi$ on $A$.
\end{remark}
\medskip

\section{An exact sequence for extensions of Lie algebras} \label{sec3}
Let $0 \to A \to L \stackrel{p}{\to} B \to 0$ be an abelian extension of Lie algebras over a field $F$ and $s:B \to L$ be a section of $p$. We show that the obstruction to inducibility of a pair of automorphisms in $\Aut(A) \times \Aut(B)$ lies in the Lie algebra cohomology $\Ha^2(B;A)$. In fact, we derive an exact sequence (see (\ref{wells-seq})) corresponding to the above extension, and relating the group of derivations $\Z^1(B;A)$, the automorphism group $\Aut_A(L)$ and the cohomology $\Ha^2(B;A)$. Here, we consider the abelian additive group structure on $\Z^1(B;A)$. The sequence is similar to the one derived for extensions of groups by Wells in \cite{Wells} and studied subsequently in \cite{Jin, Jin1, Passi, Robinson}. 

Let $\Aut^{A,B}(L)= \big\{\gamma \in \Aut(L)~|~ \tau(\gamma)=(1_A,1_B)\big\}$. Recall that
\begin{eqnarray*}
\Z^1(B;A) & = & \big\{\lambda \in \C^1(B;A)~|~ \partial^1(\lambda)=0 \big\}\\
& =& \big\{\lambda \in \C^1(B;A)~|~ \lambda\big([b_0,b_1]\big)=\big[s(b_0), \lambda(b_1)\big]-\big[s(b_1), \lambda(b_0)\big]~\textrm{for all}~b_0,b_1 \in B \big\}.
\end{eqnarray*}

\begin{proposition}
Let $\mathcal{E}:0 \to A \to L \to B \to 0$ be an abelian extension of Lie algebras over a field $F$. Then $\Z^1(B;A) \cong \Aut^{A,B}(L)$ as groups.
\end{proposition}

\begin{proof}
Define $\psi: \Z^1(B;A) \to \Aut(L)$ by $\psi(\lambda)= \gamma_\lambda$, where $\gamma_\lambda:L \to L$ is given by
$$\gamma_\lambda\big(a+s(b)\big)=a+\lambda(b)+s(b)~\textrm{for all}~a \in A~\textrm{and}~b \in B.$$ Since $s$ and $\lambda$ are $F$-linear maps, it follows that $\gamma_\lambda$ is also $F$-linear. Let $l_1=a_1+s(b_1)$ and $l_2=a_2+s(b_2)$. Then

\begin{eqnarray*}
\gamma_\lambda\big([l_1,l_2]\big)& = & \gamma_\lambda\big([a_1+s(b_1), a_2+s(b_2)]\big)\\
&= & \gamma_\lambda\big([a_1, s(b_2)]+ [s(b_1), a_2] +[s(b_1), s(b_2)]\big)\\
& =& \gamma_\lambda\big(-[s(b_2),a_1]+ [s(b_1), a_2] +\mu(b_1,b_2)+s([b_1,b_2])\big)\\
& =& -\big[s(b_2),a_1\big]+ \big[s(b_1), a_2\big] +\mu(b_1,b_2)+\lambda\big([b_1,b_2]\big)+s\big([b_1,b_2]\big)\\
& =& -\big[s(b_2),a_1\big]+ \big[s(b_1), a_2\big] +\mu(b_1,b_2)+\big[s(b_1), \lambda(b_2)\big]-\big[s(b_2), \lambda(b_1)\big]+s\big([b_1,b_2]\big)\\
& =& -\big[s(b_2),a_1\big]+ \big[s(b_1), a_2\big] +\big[s(b_1), \lambda(b_2)\big]-\big[s(b_2), \lambda(b_1)\big]+\big[s(b_1),s(b_2)\big]\\
& =& \big[\gamma_\lambda(l_1), \gamma_\lambda(l_2)\big].
\end{eqnarray*}
Thus $\gamma_\lambda$ is a Lie algebra homomorphism of $L$. Since $s$ is injective, $\gamma_\lambda(a+s(b))=0$ implies that $a=0$ and $b=0$. Finally, if $a+s(b) \in L$, then $\gamma_\lambda(a-\lambda(b)+s(b))= a+s(b)$. Hence $\gamma_\lambda$ is a Lie algebra automorphism of $L$. Clearly, $\tau (\gamma_\lambda)=(1_A,1_B)$, and hence $\psi(\gamma)=\gamma_\lambda \in \Aut^{A,B}(L)$.

Let $\lambda_1, \lambda_2 \in \Z^1(B;A)$. Then
\begin{eqnarray*}
\gamma_{\lambda_1+\lambda_2}\big(a+s(b)\big) & = & a+ \lambda_1(b)+\lambda_2(b)+s(b)\\
& = & \gamma_{\lambda_2}\big(a+ \lambda_1(b)+s(b)\big)\\
& = & \gamma_{\lambda_2}\big(\gamma_{\lambda_1}(a +s(b))\big).
\end{eqnarray*}

Thus $\psi$ is a group homomorphism. It is easy to see that $\psi$ is injective. Finally, we show that $\psi$ is surjective onto $\Aut^{A,B}(L)$. Let $\gamma \in \Aut^{A,B}(L)$. Since $\overline{\gamma}=1_B$, we have $p\big(\gamma(s(b))\big)=b=p\big(s(b)\big)$ of all $b \in B$. This implies that $\gamma\big(s(b)\big)= \lambda(b)+s(b)$ for some element $\lambda(b) \in A$. We claim that $\lambda \in \Z^1(B;A)$. Since $\gamma$ and $s$ are $F$-linear maps, it follows that $\lambda:B \to A$ is also $F$-linear. Since $\gamma$ is a Lie algebra homomorphism, we have $\gamma\big([s(b_1), s(b_2)]\big)=\big[\gamma(s(b_1)), \gamma(s(b_2))\big]$ for all $b_1, b_2 \in B$. But
\begin{eqnarray*}
\gamma\big([s(b_1), s(b_2)]\big) & = & \gamma\big(\mu(b_1, b_2)+ s([b_1, b_2])\big)\\
&=& \mu(b_1, b_2)+ \lambda\big([b_1, b_2]\big)+ s\big([b_1, b_2]\big)\\
&=& \lambda\big([b_1, b_2]\big)+ \big[s(b_1), s(b_2)\big].
\end{eqnarray*}
On the other hand, we have

\begin{eqnarray*}
\big[\gamma(s(b_1)), \gamma(s(b_2))\big] & = & \big[\lambda(b_1)+ s(b_1), \lambda(b_2)+ s(b_2)\big]\\
& = & \big[\lambda(b_1), s(b_2)\big] + \big[s(b_1), \lambda(b_2)\big]+ \big[s(b_1), s(b_2)\big].
\end{eqnarray*}
Equating these two expressions, we obtain $\lambda\big([b_1, b_2]\big)= \big[\lambda(b_1), s(b_2)\big] + \big[s(b_1), \lambda(b_2)\big]$. Hence $\lambda \in \Z^1(B;A)$. This completes the proof of the proposition.
\end{proof}

In view of the above proposition, we can view $\Z^1(B;A)$ as a subgroup of $\Aut_A(L)$. By Remark \ref{rem-compatible}, the image of $\tau$ lies in $C_\alpha$. For each $(\theta, \phi) \in \C_\alpha$, we define $\mu_{\theta, \phi}: B \otimes B \to A$ by
$$\mu_{\theta, \phi}(b_1, b_2)=\theta \big( \mu(\phi^{-1}(b_1), \phi^{-1}(b_2))\big)-\mu \big(b_1,b_2 \big)~\textrm{for}~b_1,b_2 \in B.$$
It can be easily seen that $\mu_{\theta, \phi}$ is a 2-cocycle. Since $\mu$ is the 2-cocycle corresponding to the extension $\mathcal{E}$, this defines a map $$\omega_{\mathcal{E}}: C_\alpha \to \Ha^2(B;A)$$ given by $$\omega_{\mathcal{E}} (\theta, \phi)= [\mu_{\theta, \phi}],$$ the cohomology class of $\mu_{\theta, \phi}$. 

By Theorem \ref{main-thm} and Remark \ref{rem-cocycle}, the pair $(\theta, \phi) \in C_\alpha$ is inducible if and only if $\mu_{\theta, \phi}$ is a 2-coboundary. Thus $[\mu_{\theta, \phi}] \in \Ha^2(B;A)$ is an obstruction to inducibility of the pair $(\theta, \phi) \in C_\alpha$. With the preceding discussion, we have derived the following exact sequence of groups associated to an abelian extension of Lie algebras and proving Theorem \ref{wells-sequence}.

\begin{equation}\label{wells-seq}
0 \to \Z^1(B;A) \stackrel{i}{\longrightarrow} \Aut_A(L) \stackrel{\tau}{\longrightarrow} C_ \alpha \stackrel{\omega_{\mathcal{E}}}{\longrightarrow} \Ha^2(B;A).
\end{equation}

The following are some immediate consequences of the above theorem.

\begin{corollary}
Let $\mathcal{E}:0 \to A \to L \to B \to 0$ be a split extension. Then every compatible pair is inducible.
\end{corollary}

\begin{proof}
By definition, if $0 \to A \to L \to B \to 0$ is a split extension, then $\omega_{\mathcal{E}}$ is trivial. The result now follows from the exactness of sequence (\ref{wells-seq}).
\end{proof}

\begin{corollary}
Let $0 \to A \to L \to B \to 0$ be an extension of finite dimensional Lie algebras over a field $F$ of characteristic 0. Suppose that $B$ is semisimple. Then every compatible pair is inducible.
\end{corollary}

\begin{proof}
By Whitehead's Second Lemma \cite[Proposition 6.3, p. 249]{Hilton}, we have $\Ha^2(B;A)=0$, and hence the result follows.
\end{proof}

Recall that a  finite dimensional Lie algebra $L$ is said to be perfect if $L=[L,L]$. Also, the Schur multiplier of $L$ is defined as $\mathcal{M}(L)=\Ha^2(L, F)$, where $F$ is viewed as a trivial $L$-module.

\begin{corollary}
Let $0 \to A \to L \to B \to 0$ be a central extension of finite dimensional Lie algebras over a field $F$. Suppose that $B$ is perfect and $\mathcal{M}(B)=0$. Then every pair in $\Aut(A) \times \Aut(B)$ is inducible.
\end{corollary}

\begin{proof}
Since $0 \to A \to L \to B \to 0$ is a central extension, it follows that $C_ \alpha=\Aut(A) \times \Aut(B)$. Further, $B$ being perfect and $\mathcal{M}(B)=0$ implies that $\Ha^2(B, A)=0$ by \cite[Theorem 6.12]{Batten}. Hence it follows from the exact sequence (\ref{wells-seq}) that every pair in $\Aut(A) \times \Aut(B)$ is inducible.
\end{proof}

Let $\mathcal{E}: 0 \to A \to L \to B \to 0$ be an abelian extension of Lie algebras over a field $F$  and $$\widetilde{C_ \alpha}=\big\{(\theta, \phi) \in C_ \alpha~|~\omega_{\mathcal{E}}(\theta, \phi)=0 \big\}.$$ Then the exact sequence (\ref{wells-seq}) becomes the following short exact sequence
\begin{equation}\label{short-wells-seq}
0 \to \Z^1(B;A) \stackrel{i}{\longrightarrow} \Aut_A(L) \stackrel{\tau}{\longrightarrow} \widetilde{C_ \alpha} \to 0.
\end{equation}

It is natural to investigate splitting of this short exact sequence, and we prove the following.

\begin{theorem}\label{sequence-split}
Let $0 \to A \to L \to B \to 0$ be a split and abelian extension of Lie algebras over a field $F$. Then the associated short exact sequence $0 \to \Z^1(B;A) \stackrel{i}{\longrightarrow} \Aut_A(L) \stackrel{\tau}{\longrightarrow} \widetilde{C_ \alpha} \to 0$ is also split.
\end{theorem}

\begin{proof}
Let $A\rtimes B= A \oplus B$ as a $F$-vector space and equipped with the Lie algebra structure given by
$$\big[(a_1,b_1),(a_2,b_2)\big]=\big(b_1 a_2- b_2 a_1, [b_1,b_2]\big)~\textrm{for}~a_1,a_2 \in A~\textrm{and}~b_1,b_2 \in B.$$
Then the split extension  $0 \to A \to L \to B \to 0$ is equivalent to the extension $$0 \to A \hspace*{1mm} \stackrel{a \mapsto (a,0)}{\longrightarrow} \hspace*{1mm}A\rtimes B \hspace*{1mm} \stackrel{(a,b) \mapsto b}{\longrightarrow} \hspace*{1mm}B \to 0,$$
and hence $\Aut_A(L) \cong \Aut_A(A\rtimes B)$. Note that, for a split extension, the corresponding 2-cocycle is zero, and hence $\widetilde{C_ \alpha}=C_ \alpha$. Now we define a section $\sigma: \widetilde{C_ \alpha} \to \Aut_A(A\rtimes B)$ by $\sigma(\theta, \phi)=\gamma$, where $\gamma(a,b)= \big(\theta (a), \phi(b)\big)$ for $a \in A$ and $b \in B$. Clearly, $\gamma$ is $F$-linear. Further, for $a_1,a_2 \in A$ and $b_1,b_2 \in B$, we have
\begin{eqnarray*}
\gamma\big([(a_1, b_1), (a_2,b_2)]\big) & = & \gamma\big(b_1 a_2- b_2 a_1, [b_1,b_2]\big)\\
&=& \big(\theta(b_1 a_2- b_2 a_1), \phi([b_1, b_2])\big)\\
&=& \big(\theta(b_1 a_2)- \theta(b_2 a_1), [\phi(b_1), \phi(b_2)]\big)\\
&=& \big(\phi(b_1) \theta(a_2)- \phi(b_2)\theta(a_1), [\phi(b_1),\phi(b_2)]\big)~\textrm{by compatibility of}~(\theta, \phi)\\
&=& \big[(\theta(a_1), \phi(b_1)), (\theta(a_2),\phi(b_2))\big]\\
&=& \big[\gamma(a_1, b_1), \gamma(a_2,b_2)\big].
\end{eqnarray*}
Hence $\gamma \in \Aut_A(A\rtimes B)$. It is clear that $\sigma$ is a group homomorphism, and hence the sequence (\ref{short-wells-seq}) splits.
\end{proof}

We conclude by discussing the map $\omega_{\mathcal{E}}$ in more detail. We show that there is a left action of $C_\alpha$ on $\Ha^2(B;A)$ with respect to which $\omega_{\mathcal{E}}$ is an inner derivation. Hence $\omega_{\mathcal{E}}=0$ if and only if this action of $C_\alpha$ on $\Ha^2(B;A)$ is trivial.

Let $(\theta, \phi) \in C_\alpha$ and $\nu \in \Z^2(B;A)$. For $b_1,b_2 \in B$, define
$$^{(\theta, \phi)}\nu(b_1,b_2)=\theta \big( \nu(\phi^{-1}(b_1), \phi^{-1}(b_2))\big).$$
Compatibility of $(\theta, \phi)$ implies that $^{(\theta, \phi)}\nu \in \Z^2(B;A)$. Further, if $\nu \in \B^2(B;A)$, then $\nu= \partial^1(\lambda)$ for some $\lambda \in \C^1(B;A)$. Again compatibility of $(\theta, \phi)$ implies that $^{(\theta, \phi)}\nu= \partial^1(\theta \lambda \phi^{-1})$, where $\theta \lambda \phi^{-1} \in \C^1(B;A)$. Thus $[^{(\theta, \phi)}\nu] \in \Ha^2(B;A)$. Clearly, this defines a left action of $C_\alpha$ on $\Ha^2(B;A)$ given by 
$$^{(\theta, \phi)}[\nu]= [^{(\theta, \phi)}\nu].$$

\begin{proposition}
Let $\mathcal{E}$ be an extension inducing $\alpha$. Then $\omega_{\mathcal{E}}$ is an inner derivation with respect to the action of $C_\alpha$ on $\Ha^2(B;A)$.
\end{proposition}

\begin{proof}
Let $\mathcal{E}$ be an extension inducing $\alpha$ and $\omega_{\mathcal{E}}: C_\alpha \to \Ha^2(B;A)$ be the corresponding map. Then for $(\theta, \phi)$ in $C_\alpha$, we have 

\begin{eqnarray*}
\mu_{\theta, \phi}(b_1, b_2) & = & \theta \big( \mu(\phi^{-1}(b_1), \phi^{-1}(b_2)) \big)-\mu(b_1, b_2)\\
& = & ^{(\theta, \phi)}\mu(b_1, b_2)-\mu(b_1, b_2)~\textrm{for all}~b_1,b_2 \in B.
\end{eqnarray*}
This implies $\omega_{\mathcal{E}} (\theta, \phi)= [\mu_{\theta, \phi}]= [^{(\theta, \phi)}\mu-\mu]= ^{(\theta, \phi)}[\mu]-[\mu]$. Hence $\omega_{\mathcal{E}}$ is an inner derivation with respect to the action of $C_\alpha$ on $\Ha^2(B;A)$.
\end{proof}

\begin{corollary}
Let $\alpha: B \to \End(A)$ be a $B$-module structure on $A$ and $(\theta, \phi) \in C_\alpha$. Then $(\theta, \phi)$ is inducible in each extension inducing $\alpha$ if and only if $(\theta, \phi)$ acts trivially on $\Ha^2(B;A)$.
\end{corollary}
\medskip

\section{Automorphisms of free nilpotent Lie algebras} \label{sec4}
In this section, we focus on automorphisms of free step 2 nilpotent Lie algebras. Let
$$L_{n,2}=\big\langle x_1,\dots,x_n~|~\big[[x_i,x_j],x_k\big]=0~\textrm{for all}~1 \leq i,j,k \leq n  \big\rangle$$
be the free nilpotent Lie algebra of rank $n$ and step 2. Since every Lie algebra on one generator is abelian, we assume that $n \geq 2$. 

Let $L_{n,2}^{(1)}=\big\langle [x_i,x_j]~|~1 \leq j< i \leq n  \big\rangle$ be the derived subalgebra of $L_{n,2}$. Set $Z= \{ z_{i,j}~|~z_{i,j}=[x_i,x_j]~\textrm{for}~1 \leq j< i \leq n \}$. Since $L_{n,2}$ is step 2 nilpotent, it follows that $L_{n,2}^{(1)}$ is a free abelian Lie algebra with basis $Z$ and rank $\frac{n(n-1)}{2}$. If we take the lexicographic order on the basis $Z$ given by
$$z_{2,1} < z_{3,1}< z_{3,2}< \cdots < z_{n,n-1},$$
then $\Aut\big(L_{n,2}^{(1)}\big)\cong \textrm{GL}\big(\frac{n(n-1)}{2}, F \big)$. Let $\theta \in \Aut\big(L_{n,2}^{(1)}\big)$ given by

$$
\theta : \left\{
\begin{array}{ll}
z_{2,1} \longmapsto b_{2,1;2,1} z_{2,1} + b_{2,1;3,1} z_{3,1} +\cdots + b_{2,1;n,n-1} z_{n,n-1} &  \\
\vdots \hspace{1.5cm}\vdots \hspace{3cm}\vdots \hspace{3cm}\vdots &\\
z_{i,j} \longmapsto \sum_{1 \leq l <k \leq n} b_{i,j;k,l} z_{k,l}&\\
\vdots \hspace{1.5cm}\vdots \hspace{3cm}\vdots \hspace{3cm}\vdots &\\
z_{n,n-1} \longmapsto b_{n,n-1;2,1} z_{2,1} + b_{n,n-1;3,1} z_{3,1} +\cdots + b_{n,n-1;n,n-1} z_{n,n-1}. &
\end{array} \right.
$$
Then the matrix $[\theta] \in \textrm{GL}\big(\frac{n(n-1)}{2}, F \big)$. Similarly, $L_{n,2}^{ab}=L_{n,2}/L_{n,2}^{(1)}= \big\langle \overline{x}_1,\dots, \overline{x}_n \big\rangle$ is also a free abelian Lie algebra of rank $n$, and hence $\Aut\big(L_{n,2}^{ab}\big)\cong \textrm{GL}(n, F)$. Let $\phi \in \Aut\big(L_{n,2}^{ab}\big)$ given by
$$
\phi : \left\{
\begin{array}{ll}
\overline{x}_1 \longmapsto a_{11} \overline{x}_1+ \cdots +a_{1n} \overline{x}_n &  \\
\vdots \hspace{1.5cm}\vdots \hspace{1cm}\vdots \hspace{1cm}\vdots &\\
\overline{x}_i \longmapsto a_{i1} \overline{x}_1+ \cdots +a_{in} \overline{x}_n &  \\
\vdots \hspace{1.5cm}\vdots \hspace{1cm}\vdots \hspace{1cm}\vdots &\\
\overline{x}_n \longmapsto a_{n1} \overline{x}_1+ \cdots +a_{n n} \overline{x}_n. &
\end{array} \right.
$$
Then the matrix $[\phi] \in \textrm{GL}(n, F)$. We can now formulate the main theorem of this section.

\begin{theorem}
Let  $(\theta,\phi) \in \Aut\big(L_{n,2}^{(1)}\big) \times \Aut\big(L_{n,2}^{ab}\big)$. If $[\theta]=(b_{i,j;k,l})$ and $[\phi]=(a_{ij})$ are the corresponding matrices, then the pair $(\theta, \phi)$ is inducible if and only if 
\begin{equation}\label{nil1}
a_{ik}a_{jl}- a_{il}a_{jk} =b_{i,j; k,l}~\textrm{for all}~1 \leq j <i \leq n~\textrm{and}~1 \leq l <k \leq n.
\end{equation}
\end{theorem}

\begin{proof} 
For the free nilpotent Lie algebra $L_{n,2}$, we have the following central extension
$$0 \to L_{n,2}^{(1)} \to L_{n,2} \to L_{n,2}^{ab} \to 0.$$
Let $s: L_{n,2}^{ab} \to L_{n,2}$ be the section given by $s(\overline{x}_i)=x_i$ on the generators. Then $\mu(\overline{x}_i, \overline{x}_j)= [x_i, x_j]$ for all $1 \leq i,j \leq n$. Since $L_{n,2}$ is step 2 nilpotent, it follows that RHS of Theorem \ref{main-thm}(1) is 0 and $\alpha$ of Theorem \ref{main-thm}(2) is trivial. Hence the necessary and sufficient condition of Theorem \ref{main-thm} becomes $$\theta \big( \mu(\overline{x}_i , \overline{x}_j)\big)=\mu \big(\phi(\overline{x}_i),\phi(\overline{x}_j) \big)~\textrm{for all}~1 \leq i,j \leq n.$$
In what follows, we show that this is precisely the condition (\ref{nil1}).

Recall that, $L_{n,2}$ is generated as a Lie algebra by $\{ x_1, x_2, \dots, x_n \}$. Further, the set $$\{x_1, x_2, \dots, x_n, z_{2,1}, z_{3,1}, \dots, z_{n,n-1}\}$$ is an ordered basis for $L_{n,2}$ as a vector space over $F$. Thus, if $\gamma \in \Aut(L_{n,2})$, then $\gamma$ is given by
$$
\gamma : \left\{
\begin{array}{ll}
x_1 \longmapsto \alpha_{11} x_1+\alpha_{12} x_2+ \cdots +\alpha_{1n} x_n + \beta_{1;2,1} z_{2,1} + \beta_{1;3,1} z_{3,1} +\cdots + \beta_{1;n,n-1} z_{n,n-1} &  \\
\vdots \hspace{1.5cm}\vdots \hspace{5cm}\vdots \hspace{5cm}\vdots &\\
x_n \longmapsto \alpha_{n1} x_1+\alpha_{n2} x_2+ \cdots +\alpha_{nn} x_n + \beta_{n;2,1} z_{2,1} + \beta_{n;3,1} z_{3,1} +\cdots + \beta_{n;n,n-1} z_{n,n-1} &
\end{array} \right.
$$
for some $\alpha_{ij}, \beta_{i;k,l} \in F$.
Now, suppose that $(\theta, \phi)$ is inducible by $\gamma$. Then 
$$
\overline \gamma : \left\{
\begin{array}{ll}
\overline{x}_1 \longmapsto \alpha_{11} \overline{x}_1+\alpha_{12} \overline{x}_2+ \cdots +\alpha_{1n} \overline{x}_n &  \\
\vdots \hspace{1.5cm}\vdots \hspace{1.7cm}\vdots \hspace{1.7cm}\vdots &\\
\overline{x}_n \longmapsto \alpha_{n1} \overline{x}_1+\alpha_{n2} \overline{x}_2+ \cdots +\alpha_{nn} \overline{x}_n. &
\end{array} \right.
$$
Since $\overline \gamma= \phi$, we obtain 
\begin{equation}\label{nil2}
\alpha_{ij}=a_{ij}~\textrm{for all}~ 1 \leq i, j \leq n.
\end{equation}

Next, we consider $\gamma|_{L_{n,2}^{(1)}}$ as follows
\begin{eqnarray*}
\gamma(z_{i,j}) & = & \gamma \big([x_i,x_j]\big)\\
& = & \big[\gamma(x_i), \gamma(x_j)\big]\\
&= & \big[\alpha_{i1} x_1+\alpha_{i2} x_2+ \cdots +\alpha_{in} x_n + \beta_{i;2,1} z_{2,1} + \beta_{i;3,1} z_{3,1} +\cdots + \beta_{i;n,n-1} z_{n,n-1},\\ 
& & \alpha_{j1} x_1+\alpha_{j2} x_2+ \cdots +\alpha_{jn} x_n + \beta_{j;2,1} z_{2,1} + \beta_{j;3,1} z_{3,1} +\cdots + \beta_{j;n,n-1} z_{n,n-1} \big]\\
&= & \big[\alpha_{i1} x_1+\alpha_{i2} x_2+ \cdots +\alpha_{in} x_n, \alpha_{j1} x_1+\alpha_{j2} x_2+ \cdots +\alpha_{jn} x_n],~\textrm{since $L_{n,2}$ is step 2 nilpotent}\\
&= & \sum_{1 \leq l < k \leq n} \big(\alpha_{ik}\alpha_{jl}- \alpha_{il}\alpha_{jk} \big) z_{k,l}.
\end{eqnarray*}
Since $\gamma|_{L_{n,2}^{(1)}}= \theta$, we obtain 
\begin{equation}\label{nil3}
\alpha_{ik}\alpha_{jl}- \alpha_{il}\alpha_{jk} =b_{i,j; k,l} ~\textrm{for all}~1 \leq j <i \leq n~\textrm{and}~1 \leq l <k \leq n.
\end{equation}
Combining equations (\ref{nil2}) and (\ref{nil3}), we get (\ref{nil1}).
\end{proof}

As a consequence, we obtain the following result for $n=2$.

\begin{corollary}\label{n2k2}
Let  $(\theta,\phi) \in \Aut\big(L_{2,2}^{(1)}\big) \times \Aut\big(L_{2,2}^{ab}\big)$ and $\big(r, [\phi]\big)$ be the corresponding pair of matrices. Then the pair $(\theta, \phi)$ is inducible if and only if $r=\det [\phi]$.
\end{corollary}

\begin{proof} 
In this case, the derived subalgebra $L_{2,2}^{(1)}$ is a one-dimensional free abelian algebra generated by $[x_2,x_1]$ and the abelianization $L_{2,2}^{ab}$ is a two-dimensional free abelian algebra with basis $\{\overline{x}_1,\overline{x}_2\}$. The proof now follows from (\ref{nil1}).
\end{proof}

We conclude by giving an example to show that the converse of Theorem \ref{sequence-split} is not true in general.

\begin{example}
Consider the central extension of Lie algebras 
$$0 \to L_{2,2}^{(1)} \to L_{2,2} \to L_{2,2}^{ab} \to 0.$$
Since $L_{2,2}$ is non-abelian, the sequence does not split. We show that the associated short exact sequence
\begin{equation}\label{exple}
0 \to \Z^1\big(L_{2,2}^{ab}; L_{2,2}^{(1)} \big) \stackrel{i}{\longrightarrow} \Aut_{L_{2,2}^{(1)}}\big(L_{2,2}\big) \stackrel{\tau}{\longrightarrow} \widetilde{C_ \alpha} \to 0
\end{equation} 
splits. We define a section $\sigma: \widetilde{C_ \alpha}  \to 
\Aut_{L_{2,2}^{(1)}}\big(L_{2,2}\big)$ which is a group homomorphism, showing that the sequence (\ref{exple}) splits. Define $\sigma(\theta, \phi)= \gamma$, where
$$
\gamma : \left\{
\begin{array}{ll}
x_1 \longmapsto s\big(\phi( \overline{x}_1)\big) &  \\
x_2 \longmapsto s\big(\phi( \overline{x}_2)\big) & \\
$[$ x_1,x_2] \longmapsto \theta\big([x_1,x_2]\big).
\end{array} \right.
$$
Then 
\begin{eqnarray*}
\big[\gamma(x_1), \gamma(x_2) \big] & = & \big[s\big(\phi( \overline{x}_1)\big), s\big(\phi( \overline{x}_2)\big) \big]\\
&=& \big[s(a_{11}\overline{x}_1+a_{12}\overline{x}_2), s(a_{21}\overline{x}_1+ a_{22}\overline{x}_2) \big]\\
&=& \big[a_{11}s(\overline{x}_1)+a_{12}s(\overline{x}_2), a_{21}s(\overline{x}_1)+ a_{22}s(\overline{x}_2) \big]\\
&=& \big[a_{11}x_1+a_{12}x_2, a_{21}x_1+ a_{22}x_2 \big]\\
&=& (a_{11}a_{22}- a_{12}a_{21}) [x_1,x_2]\\
&=& \theta\big( [x_1,x_2] \big)~\textrm{by Corollary \ref{n2k2}, since $(\theta, \phi)$ is inducible}\\
&=& \gamma\big( [x_1,x_2] \big).\\
\end{eqnarray*}
It follows that $\gamma \in \Aut_{L_{2,2}^{(1)}}\big(L_{2,2}\big)$. Since $\tau(\gamma)= (\theta, \phi)$, $\sigma$ is a section. It is easy to see that $\sigma$ is a group homomorphism, and hence the sequence (\ref{exple}) splits.
 \end{example}

\medskip

\medskip \noindent \textbf{Acknowledgement.} The authors thank the referee for many useful comments and for the reference \cite{Bahturin}. Support from the DST-RSF project INT/RUS/RSF/2 is gratefully acknowledged. Bardakov is partially supported by Laboratory of Quantum Topology of Chelyabinsk State University, and grants RFBR-16-01-00414,  RFBR-14-01-00014 and RFBR-15-01-00745. Singh is also supported by DST INSPIRE Scheme IFA-11MA-01/2011 and DST Fast Track Scheme SR/FTP/MS-027/2010.

\end{document}